\documentclass[10pt,reqno]{amsart}
\usepackage{amsmath,amsfonts,amsthm,amssymb}
\usepackage[shortlabels]{enumitem}
\usepackage{tikz}
\usepackage{soul}
\usepackage{hyperref}

\newcommand{\R}{\mathbb{R}}
\newcommand{\C}{\mathbb{C}}
\newcommand{\N}{\mathbb{N}}
\newcommand{\id}{\operatorname{id}}

\newcommand{\cl}{\overline}

\newcommand{\inter}[1]{{#1}^\circ}
\newcommand{\orbit}{\mathcal{O}}
\newcommand{\Fix}{\operatorname{Fix}}

\newcommand{\diam}{\operatorname{diam}}

\newcommand{\bs}{\setminus}

\theoremstyle{plain}
\newtheorem{theorem}{Theorem}[section]
\newtheorem{lemma}[theorem]{Lemma}

\newtheorem{corollary}[theorem]{Corollary}
\newtheorem{proposition}[theorem]{Proposition}

\theoremstyle{remark}
\newtheorem{remark}[theorem]{Remark}

\theoremstyle{definition}

\numberwithin{equation}{section}

\begin{document}
\title[Fixed points of Thompson metric nonexpansive maps]{Bounded fixed point sets and Krasnoselskii iterates of Thompson metric nonexpansive maps}
\author[B. Lins]{Brian Lins}
\date{}
\address{Brian Lins, Hampden-Sydney College}
\email{blins@hsc.edu}

\begin{abstract}
We consider maps defined on the interior of a normal, closed cone in a real Banach space that are nonexpansive with respect to Thompson's metric. With mild compactness assumptions, we prove that the Krasnoselskii iterates of such maps converge to a fixed point when one exists. For maps that are also order-preserving, we give simple necessary and sufficient conditions in terms of upper and lower Collatz-Wielandt numbers for the fixed point set to be nonempty and bounded in Thompson's metric. When the map is also real analytic, these conditions are both necessary and sufficient for the map to have a unique fixed point and for all iterates of the map to converge to the fixed point. We demonstrate how these results apply to certain nonlinear matrix equations on the cone of positive definite Hermitian matrices.
\end{abstract}

\maketitle
 
\section{Introduction}

%

Let $C$ be a normal, closed cone with nonempty interior $\inter{C}$ in a Banach space. Thompson's \cite{Thompson63} metric $d_T$ is a complete metric on $\inter{C}$. A map $f: \inter{C} \rightarrow \inter{C}$ is \emph{nonexpansive} with respect to Thompson's metric if $d_T(f(x),f(y)) \le d_T(x,y)$ for all $x, y \in \inter{C}$. Thompson metric nonexpansive maps arise in a variety of applications involving cones such as the standard cone $\R^n_{\ge 0}$ of nonnegative vectors in $\R^n$ \cite{CaLiSt19,GaSt20,Gnacadja07}, cones of positive semidefinite matrices \cite{DuWaLi13, LaLi12, Lim09, WeSr22}, cones of nonnegative functions in various function spaces \cite{Krasnoselskii64,Thompson63}, and cones of positive operators \cite{LaLi08}.  

A standard method to find a fixed point of a nonexpansive map $f$ is to iterate a linear relaxation of $f$: $\alpha f + (1- \alpha) \id$ where $0 < \alpha < 1$ and $\id$ is the identity map.  
Krasnoselskii showed \cite{Krasnoselskii55} that if a compact nonexpansive map has a fixed point in a uniformly convex Banach space, then the iterates of $\tfrac{1}{2}(f + \id)$ converge to a fixed point. Schaefer \cite{Schaefer57} extended this result to iterates of $\alpha f + (1-\alpha) \id$ for any $0< \alpha < 1$.  Later, Edelstein and O'Brien \cite{EdOB78} and Ishikawa \cite{Ishikawa76} independently proved that Krasnoselskii iteration does not require uniform convexity and in fact can be applied to nonexpansive maps on any Banach space, after making suitable compactness assumptions. 
Ishikawa's results also apply to more general Krasnoselskii-Mann iterations where the constants $\alpha$ are allowed to vary. 

Kirk \cite{Kirk82} showed that a natural generalization of Krasnoselskii iteration converges to fixed points for nonexpansive maps on a large class of hyperbolic spaces which includes normed spaces and CAT(0) spaces. However, general Thompson geometries are not hyperbolic spaces under the definition used in \cite{Kirk82}, nor do they fit into other classes of hyperbolic spaces where Krasnoselskii-Mann iteration has been studied \cite{Leustean10,ReSh90}.  

Iterative methods to find fixed points of Thompson's metric nonexpansive maps typically require extra assumptions so that the maps are contractions \cite{KrNu93}, or involve some kind of nonlinear relaxation chosen for a specific cone \cite[Theorem 11]{GaSt20}, \cite{Herzog14}.  In Section \ref{sec:Kras}, we show that if a Thompson metric nonexpansive map $f$ has a fixed point in $\inter{C}$, then under mild compactness assumptions which are always satisfied in finite dimensions, the iterates of the linear relaxation $\alpha f+ (1-\alpha) \id$ converge to a fixed point of $f$ for any $0 < \alpha < 1$ and any starting point in $\inter{C}$. The compactness assumptions are stated in terms of Kuratowski's measure of noncompactness, which can be defined using either the norm or Thompson's metric. The two metrics yield different versions of Kuratowski's measure of noncompactness, each with different properties. The properties of the norm version, which we denote by $\gamma$, are well-known (see e.g., \cite{Deimling}), and the properties of the version $\tau$ defined in terms of Thompson metric were investigated in \cite{HeKu14}. We extend these properties by proving in Theorem \ref{thm:condensing-nonexpansive} that the sum of a $\tau$-condensing map with a $d_T$-nonexpansive map is $\tau$-condensing. 

In Section \ref{sec:Order}, we consider Thompson metric nonexpansive maps which are also order-preserving with respect to the partial order induced by the cone. Much of what is known about the fixed point theory for such maps is due to the work of Krasnoselskii and his students Ladyzhenskii and Bahktin (see \cite[Chapter 6]{Krasnoselskii64} and the bibliographic notes there for details), and to Thompson \cite{Thompson63}. Some additional recent results can be found in \cite{CaLiSt19,MaMo19}. 

In Theorem \ref{thm:bounded}, we show that a simple sufficient condition for the existence of fixed points, involving upper and lower Collatz-Wielandt numbers, is also necessary for the fixed point set to be bounded in Thompson's metric. Although the sufficient condition for the existence of fixed points in Theorem \ref{thm:bounded} appears to be new, the main novelty of Theorem \ref{thm:bounded} is the proof that these conditions are necessary for the fixed point set to be bounded.  

In Section \ref{sec:Analytic} we focus on Thompson's metric nonexpansive maps which are also real analytic. We show in Theorem \ref{thm:analytic2} that the conditions of Theorem \ref{thm:bounded} are necessary and sufficient for an order-preserving, real analytic, $d_T$-nonexpansive map $f$ to have a unique fixed point in $\inter{C}$, and in that case the iterates $f^k(x)$ converge to the fixed point for every initial point $x \in \inter{C}$. We also prove in Theorem \ref{thm:analytic1} that if $C$ is a finite dimensional closed cone and $f:\inter{C} \rightarrow \inter{C}$ is real analytic and $d_T$-nonexpansive map (but not necessarily order-preserving), then $f$ has a nonempty and bounded set of fixed points in $\inter{C}$ if and only if $f$ has a unique fixed point.  


Although most of the results in this paper are stated for infinite dimensional Thompson geometries, the main results are new and noteworthy even in finite dimensions. 
For an example, we demonstrate in Section \ref{sec:matrix} how the results of Sections \ref{sec:Order} and \ref{sec:Analytic} can be applied to find solutions of a class of nonlinear matrix equations on the cone of positive definite matrices.

\section{Preliminaries} \label{sec:Prelim}

In what follows, we use $\inter{A}, \cl{A}$, and $\partial A$ to denote respectively the interior, closure, and boundary of a set $A$. We let $\Fix(f)$ denote the set of fixed points of a function $f$. 

\subsection{Cones and Thompson's metric}
Let $X$ be a real Banach space with norm $\| \cdot \|$ and dual space $X^*$.  A \emph{closed cone} is a closed convex set $C \subset X$ such that (i) $\lambda C \subseteq C $ for all $\lambda \ge 0$ and (ii) $C \cap (-C) = \{0\}$. A closed cone $C$ induces the following partial order on $X$.  We say that $x \le y$ whenever $y - x \in C$. We will write $x \ll y$ when $y-x \in \inter{C}$. For $x,y \in C$, we let 
$$M(x/y) = \inf \{ \beta > 0 : x \le \beta y \}$$
and
$$m(x/y) = \sup \{ \alpha > 0 : \alpha y \le x \}.$$
An alternative formula for $M(x/y)$ when $y \in \inter{C}$ (see e.g., \cite[Lemma 2.2]{LLNW18}) is:
\begin{equation} \label{functionals}
M(x/y) = \sup_{\phi \in C^*} \frac{\phi(x)}{\phi(y)}
\end{equation}
where $C^* = \{\phi \in X^* : \phi(x) \ge 0 \text{ for all } x \in C \}$ is the \emph{dual cone} of $C$. 

Two elements $x, y \in X$ are \emph{comparable}, denoted $x \sim y$, if there are constants $\alpha, \beta > 0$ such that $\alpha x \le y \le \beta x$. Comparability is an equivalence relation on $C$, and the equivalence classes are called the \emph{parts} of $C$. If $C$ has nonempty interior, then $\inter{C}$ is a part. For comparable $x, y \in C$, \emph{Thompson's metric} is 
$$d_T(x,y) = \log \left( \max \{ M(x/y), M(y/x)\} \right) = \log \inf \{ \beta \ge 1 : \beta^{-1} x \le y \le \beta x \}.$$
We use $B_R(x) = \{ y \in C : y \sim x,~ d_T(x,y) < R \}$ to denote the open balls in Thompson's metric. For any $x, y \in X$, we let $[x,y]$ denote the \emph{order interval} 
$$[x,y] = \{z \in X : x \le z \le y \}.$$ 
Observe that for any $x \in C$, 
\begin{equation} \label{balls}
\cl{B_R(x)} = [e^{-R}x, e^R x].
\end{equation}

A closed cone $C$ in a Banach space $X$ is \emph{normal} if there is a constant $\kappa$ such that $\|x\| \le \kappa \|y\|$ whenever $0 \le x \le y$. When $C$ is normal, Thompson's metric is a complete metric on each part $C$, and the topology induced by Thompson's metric is equivalent to the norm topology \cite[Lemma 3]{Thompson63}. It is apparent from the definition that Thompson metric balls are bounded in the norm topology when $C$ is a normal cone. 

Let $x_n$ be a sequence in a Banach space $X$ with a partial order induced by a closed cone $C$. We say that $x_n$ is \emph{decreasing} (\emph{increasing}) if $x_{n+1} \le x_n$ ($x_{n+1} \ge x_n$) for all $n \in \N$.  A closed cone $C$ is \emph{regular} if every decreasing sequence in $C$ converges. All finite dimensional cones are regular and all regular cones are normal \cite[Proposition 19.2]{Deimling}.  

Many important closed cones in infinite dimensional Banach spaces have empty interior.  For example, the cone $L^p([a,b])_{\ge 0}$ of almost everywhere nonnegative functions in $L^p([a,b])$ has empty interior for all $1 \le p < \infty$. However, $L^p([a,b])_{\ge 0}$ is normal since $0 \le f \le g$ implies that $\|f\|_p \le \|g\|_p$. In fact, $L^p([a,b])$ is also regular by the monotone convergence theorem. 

Suppose that $C$ is a closed cone in a Banach space $X$. Let $u \in C$ and let $C_u$ be the part of $C$ containing $u$. Let 
\begin{equation} \label{order-unit}
X_u = \bigcup_{k > 0} [-ku,ku] ~\text{ and }~ \|x\|_u = \inf \{k > 0 : x \in [-ku,ku]\}.
\end{equation}
If $C$ is a normal cone, then $(X_u,\|\cdot\|_u)$ is a Banach space which is continuously embedded in $(X,\|\cdot\|)$ \cite[Proposition 19.9]{Deimling}. Furthermore, $\overline{C_u} = C \cap X_u$ is a normal, closed cone with interior equal to $C_u$ in $(X_u,\|\cdot\|_u)$. 
For this reason, when we are interested in the Thompson geometry on a part of a normal closed cone, we can assume without loss of generality that the part is the interior of a normal cone in some Banach space.

The following observation will be used in Section \ref{sec:Kras}.

\begin{lemma} \label{lem:falpha}
Let $C$ be a closed cone with nonempty interior in a Banach space. If $f: \inter{C} \rightarrow \inter{C}$ is $d_T$-nonexpansive, then for any $0 < \alpha < 1$ the map $\alpha f + (1 - \alpha) \id$ is also nonexpansive on $(\inter{C},d_T)$. 
\end{lemma}

\begin{proof}
Fix $x, y \in \inter{C}$. By definition, $d_T(x,y) = \log \inf \{ \beta \ge 1 : \beta^{-1} x \le y \le \beta x \}$. 
Suppose that $d_T(x,y) = \log \beta$ for some $\beta > 1$.  Since $C$ is closed, this means that 
$$\beta^{-1} x \le y \le \beta x.$$
Since $f$ is $d_T$-nonexpansive, we also have $d_T(f(x),f(y)) \le \log \beta$ so  
$$\beta^{-1} f(x) \le f(y) \le \beta f(x)$$
Therefore 
$$\beta^{-1} \left( \alpha f(x) + (1-\alpha) x \right) \le \alpha f(y) + (1-\alpha) y \le \beta \left( \alpha f(x) + (1-\alpha) x\right).$$
This means that
$$d_T(\alpha f(x) + (1-\alpha) x, \alpha f(y) + (1-\alpha) y) \le \log \beta = d_T(x,y),$$
and therefore $\alpha f + (1-\alpha) \id$ is nonexpansive. 
\end{proof}

\subsection{Nonexpansive maps and omega limit sets}

Let $(M,d)$ be a metric space.  Let $f: M \rightarrow M$ be nonexpansive.  The \emph{orbit} of $x \in M$ under $f$ is $\orbit(x,f) = \{f^k(x) : k \in \N\}$.  The \emph{omega limit set} of $x$ under $f$ is
$$\omega(x,f) = \bigcap_{n \in \N} \cl{\{ f^k(x) : k \ge n \}}.$$

The following result combines \cite[Lemma 3.1.2 and Corollary 3.1.5]{LemmensNussbaum}.  It can be traced back to a theorem of Freudenthal and Hurewicz which states that any surjective nonexpansive map from a compact metric space onto itself must be an isometry \cite{FrHu36}.  

\begin{proposition} \label{prop:isometry}
Let $(M,d)$ be a metric space.  If $f: M \rightarrow M$ is nonexpansive and $\orbit(x,f)$ has compact closure in $M$, then $\omega(x,f)$ is a nonempty compact set and $f$ restricted to $\omega(x,f)$ is an invertible isometry. 
\end{proposition}

\subsection{Measures of noncompactness}

Let $X$ be a Banach space and let $\mathcal{B}$ denote the bounded subsets of $X$. \emph{Kuratowski's measure of noncompactness} on $(X, \|\cdot\|)$ is the function $\gamma: \mathcal{B} \rightarrow [0,\infty)$ defined by 
$$\gamma(A) = \inf \{d > 0 : A \text{ admits a finite cover by sets with diameter} \le d \}.$$
This measure of noncompactness has several properties (see e.g., \cite[Proposition 7.2]{Deimling}), of which we note the following. For any $A, B \in \mathcal{B}$ and $\lambda \in \R$, 
\begin{itemize}
\item[P1.] \label{item:compact} $\gamma(A) = 0$ if and only if $\cl{A}$ is compact. 
\item[P2.] \label{item:seminorm} $\gamma$ is a seminorm, i.e., $\gamma(\lambda A) = |\lambda| \gamma(A)$ and $\gamma(A+B) \le \gamma(A)+\gamma(B)$.
\end{itemize}

Let $D \subseteq X$ and let $f: D \rightarrow X$ be continuous.  We say that $f$ is $\gamma$-\emph{condensing} if $\gamma(f(A)) < \gamma(A)$ whenever $A \subseteq D$ is bounded and $\gamma(A) > 0$. By the Darbo-Sadovski\u{\i} theorem (see e.g., \cite[Theorem 9.1]{Deimling}), if $C$ is a nonempty, closed, bounded, convex subset of $D$, and $f$ is $\gamma$-condensing with $f(C) \subseteq C$, then $f$ has a fixed point in $C$.  

The following lemma is a well known application of property (P2) which we will use in the sequel.
\begin{lemma} \label{lem:falpha2}
Let $X$ be a Banach space and let $D \subseteq X$. If $f: D \rightarrow X$ is $\gamma$-condensing, then for any $0 < \alpha < 1$ the map $\alpha f + (1 - \alpha) \id$ is also $\gamma$-condensing.
\end{lemma}

\begin{proof}
Let $A$ be a bounded subset of $D$. By property (P2) of $\gamma$, 
\begin{align*}
\gamma(\alpha f(A) + (1-\alpha) A) &\le \gamma(\alpha f(A)) + \gamma((1-\alpha)A) \\
&\le \alpha \gamma(f(A)) + (1-\alpha)\gamma(A) \\
&< \gamma(A). & (\text{since }f\text{ is }\gamma\text{-condensing})
\end{align*}
Therefore $\alpha f + (1-\alpha) \id$ is $\gamma$-condensing.
\end{proof}

Let $C$ be a closed cone with nonempty interior in a Banach space. 
Let $\mathcal{B}_T$ denote the $d_T$-bounded subsets of $\inter{C}$ and let $\diam_T(A)$ denote the Thompson's metric diameter of $A \in \mathcal{B}_T$. 
\emph{Kuratowski's measure of noncompactness} on $(\inter{C},d_T)$ is
$$\tau(A) = \inf \{ d > 0: A \text{ has a finite cover by sets } A_i \text{ with } \diam_T(A_i) \le d \}.$$
The properties of $\tau$ were investigated by Herzog and Kunstmann \cite{HeKu14}. Although $\tau$ satisfies property (P1) of $\gamma$, it does not satisfy (P2). Instead by \cite[Proposition 2.5]{HeKu14}, 
\begin{itemize}
\item[P3.] $\tau(\lambda A) = \tau(A)$ for all $\lambda > 0$ and $\tau(A+B) \le \max \{ \tau(A), \tau(B) \}$,  
\end{itemize}
for any $A, B \in \mathcal{B}_T$.

Let $D$ be a subset of $\inter{C}$ and let $f:D \rightarrow \inter{C}$ be continuous. Then $f$ is \emph{$\tau$-condensing} if $\tau(f(A)) < \tau(A)$ for every $d_T$-bounded subset $A \subset D$ such that $\tau(A) > 0$. One advantage of $\tau$ over $\gamma$ is that if $f$ is $\tau$-condensing, then so is $c f$ for all $c > 0$. If $D$ is a nonempty closed, convex, $d_T$-bounded subset of $\inter{C}$, and $f:D \rightarrow D$ is continuous and $\tau$-condensing, then $f$ has a fixed point in $D$ \cite[Theorem 4.1]{HeKu14}. Here we will prove that the sum of a $\tau$-condensing map and a $d_T$-nonexpansive map is $\tau$-condensing.  

\begin{lemma} \label{lem:plus-u}
Let $C$ be a normal, closed cone with nonempty interior in a Banach space. The map $f(x) = x + u$ is $\tau$-condensing on $\inter{C}$ for any $u \in \inter{C}$. 
\end{lemma}
\begin{proof}
On any $d_T$-bounded set $A \subset \inter{C}$, the map $f$ is a strict $d_T$-contraction, i.e., there is a constant $0 < c < 1$ such that $d_T(f(x),f(y)) \le c d_T(x,y)$ for all $x, y \in A$ \cite[Theorem 5.3]{LaLi12}.  If $\tau(A) = \delta > 0$, then for any $\epsilon > 0$, $A$ can be covered by a finite collection of sets $A_1, \ldots, A_n \subseteq A$ each with $\diam_T(A_i) \le \delta + \epsilon$. Then $\diam_T(f(A_i)) \le c \diam_T(A_i) \le c (\delta+ \epsilon)$, so $\tau(f(A)) \le c(\delta + \epsilon)$ for any $\epsilon > 0$. Therefore $\tau(f(A)) \le c \delta < \delta = \tau(A)$.
%
\end{proof}

\begin{theorem} \label{thm:condensing-nonexpansive}
Let $C$ be a normal, closed cone with nonempty interior in a Banach space. If $f: \inter{C} \rightarrow \inter{C}$ is $\tau$-condensing and $g:\inter{C} \rightarrow \inter{C}$ is $d_T$-nonexpansive, then $f+g$ is $\tau$-condensing. 
\end{theorem}

\begin{proof}
Let $A \subset \inter{C}$ be bounded in Thompson's metric and suppose that $\tau(A) = \delta > 0$. Then $\tau(f(A)) < \delta$ and $\tau(g(A)) \le \delta$ since $f$ is $\tau$-condensing and $g$ is nonexpansive. 

Since $f(A)$ is $d_T$-bounded, we can choose $u \in \inter{C}$ such that $x \ge u$ for all $x \in f(A)$. Then for every $x \in f(A)$ and $0 < c < 1$, 
$$x - cx \le x - cu \le x.$$
This implies that $d_T(x, x - cu) \le |\log (1-c)|$. 

Since $\tau(f(A)) < \delta$, we can cover $f(A)$ with a finite collection of sets $B_1, \ldots, B_n \subset \inter{C}$ each with $\diam_T(B_i) \le \delta'$ where $\delta' < \delta$. Now consider $\diam_T(B_i - cu)$.  Observe that if $x, y \in B_i - cu$, then by the triangle inequality
\begin{align*}
d_T(x,y) &\le d_T(x,x+cu) + d_T(x+cu,y+cu) + d_T(y+cu, y) \\
&\le \delta' + 2 |\log(1 - c)|.
\end{align*}
By choosing $c > 0$ small enough, we can guarantee that $\diam_T(B_i - cu) < \delta$ for every $B_i$, so $\tau(f(A)-cu) < \delta$. By Lemma \ref{lem:plus-u}, $\tau(g(A)+cu) < \delta$.  Then by property (P3) of $\tau$,
\begin{align*}
\tau(f(A) + g(A)) &= \tau(f(A) - cu + g(A) + cu) \\
&\le \max \{\tau(f(A) - cu), \tau(g(A)+cu) \} < \delta.
\end{align*}
Therefore $f+g$ is $\tau$-condensing.
\end{proof} 

\begin{corollary} \label{cor:falpha3}
Let $C$ be a normal, closed cone with nonempty interior in a Banach space. If $f: \inter{C} \rightarrow \inter{C}$ is $\tau$-condensing, then so is $\alpha f + (1- \alpha) \id$ for every $0 < \alpha < 1$.  
\end{corollary}

\begin{proof}
Observe that $\alpha f$ is $\tau$-condensing by property (P3) and $(1-\alpha) \id$ is a Thompson metric isometry, so it is nonexpansive.  Therefore $\alpha f + (1-\alpha)\id$ is $\tau$-condensing by Theorem \ref{thm:condensing-nonexpansive}.
\end{proof}

\section{Krasnoselskii iteration} \label{sec:Kras}

\begin{theorem} \label{thm:Krasnoselskii}
Let $C$ be a normal, closed cone with nonempty interior in a Banach space. Let $f: \inter{C} \rightarrow \inter{C}$ be $d_T$-nonexpansive and either $\gamma$ or $\tau$-condensing. If $f$ has a fixed point in $\inter{C}$, then for any $0 < \alpha < 1$ and $x_0 \in \inter{C}$, the sequence $x_n$ defined by 
$$x_{n+1} = \alpha f(x_n) + (1-\alpha) x_n$$
converges to a fixed point of $f$. 
\end{theorem}

\begin{proof} 
Let $g = \alpha f + (1-\alpha) \id$.  Note that $g$ is nonexpansive by Lemma \ref{lem:falpha} and either $\gamma$-condensing by Lemma \ref{lem:falpha2} or $\tau$-condensing by Corollary \ref{cor:falpha3}. Since $f$ and therefore $g$ both have a fixed point in $\inter{C}$, the orbit $\orbit(x,g)$ is bounded in $(\inter{C}, d_T)$ and therefore also in the norm topology.  
Since $g$ is $\gamma$ or $\tau$-condensing, it follows that $\orbit(x,g)$ has compact closure by property (P1). Therefore the omega limit set $\omega(x,g)$ is a nonempty compact subset of $\inter{C}$.

Choose $y, z \in \omega(x,g)$ such that $d_T(y,z)$ is maximal.  We may assume without loss of generality that $d_T(y,z) = \log M(z/y)$. Then there is a linear functional $\phi \in C^*$ such that $M(z/y) = {\phi(z)}/{\phi(y)}$ (see \cite[Lemma 2.2]{LLNW18}).  Let $a = \phi(y)$ and $b = \phi(z)$ and note that $b \ge a$ since $0 \le d_T(y,z) = \log(b/a)$. 

By Proposition \ref{prop:isometry}, $g$ is an invertible isometry on $\omega(x,g)$. Let $g^{-1}$ denote the inverse of $g$ on $\omega(x,g)$ and let $y^{-1} = g^{-1}(y)$ and $z^{-1} = g^{-1}(z)$. 
Observe that $a \le \phi(w) \le b$ for all $ w \in \omega(x,g)$, otherwise $d_T(w,z)$ or $d_T(y,w)$ would be greater than $\log(b/a)$ by \eqref{functionals}, but $\log(b/a)$ is the maximal distance between pairs in $\omega(x,g)$. In particular, $\phi(y^{-1}) \ge a$ and $\phi(z^{-1}) \le b$.   

Since $\phi(y) = a$ and $y = g(y^{-1})$ is a convex combination of $y^{-1}$ and $f(y^{-1})$, it follows that $\phi(f(y^{-1})) \le a$.  Similarly $\phi(f(z^{-1})) \ge b$.  However, 

\begin{align*}
\log \frac{b}{a} &\le \log \frac{\phi(f(z^{-1}))}{\phi(f(y^{-1}))} \\
&\le d_T(f(y^{-1}),f(z^{-1}))  \\
&\le d_T(y^{-1},z^{-1}) & (\text{nonexpansiveness})\\
&= \log \frac{b}{a}. & (\text{since }g \text{ is an isometry on } \omega(x,g))
\end{align*}
We conclude that $\phi(y^{-1}) = a$ and $\phi(z^{-1}) = b$.

We can repeat this argument to prove that $\phi(z^{-k}) = b$ for all $k \in \N$ where $z^{-k} = g^{-k}(z) \in \omega(x,g)$. However, there is a point $g^m(x) \in \mathcal{O}(x,g)$ that is arbitrarily close to $y$ and an $n \in \N$ such that $g^{m+n}(x)$ is arbitrarily close to $z$.  Then $g^n(y)$ will be arbitrarily close to $z$ by the nonexpansiveness of $g$.  Since $g$ is an isometry on $\omega(x,g)$, we have $d_T(g^n(y),z) = d_T(y,z^{-n})$ arbitrarily small. But since $\phi(y) = a$ and $\phi(z^{-k}) = b$, we have $d_T(y,z^{-k}) \ge \log (b/a)$ for all $k \in \N$, which is a contradiction unless $a=b$ and $\omega(x,g)$ is a singleton.  
\end{proof}

A \emph{retraction} is a continuous map $r$ from a topological space $X$ to a subspace $Y \subseteq X$ such that $r(X) = Y$ and $r$ restricted to $Y$ is the identity map.  The set $Y$ is called a \emph{retract} of $X$. If $r$ is also nonexpansive, then $Y$ is called a \emph{nonexpansive retract}. Bruck showed that if a nonexpansive map on a Banach space satisfies a conditional fixed point property, then its fixed point set is a nonexpansive retract \cite{Bruck73}. Later, Nussbaum showed that the fixed point set of a Thompson metric nonexpansive map is a nonexpansive retract under certain compactness assumptions \cite[Theorem 4.7 and Corollary 4.1]{Nussbaum88}. Here we use Theorem \ref{thm:Krasnoselskii} to give a simple proof that the fixed point set of a condensing Thompson metric nonexpansive map is a nonexpansive retract. 

\begin{corollary} \label{cor:contractible}
Let $C$ be a normal, closed cone with nonempty interior in a Banach space. If $f: \inter{C} \rightarrow \inter{C}$ is $d_T$-nonexpansive and either $\gamma$ or $\tau$-condensing, then $\Fix(f)$ is a nonexpansive retract of $\inter{C}$.  
\end{corollary}

\begin{proof}
Suppose that $\Fix(f)$ is nonempty. Let $r$ be the map which assigns to any $x \in \inter{C}$ the limit of the sequence $x_n$ defined by $x_0 = x$ and $x_{n+1} = \alpha f(x_n) + (1-\alpha) x_n$ with $0 < \alpha < 1$.  Then $r$ is a nonexpansive retraction and its range is $\Fix(f)$. 
\end{proof}

\section{Bounded fixed point sets of order-preserving maps} \label{sec:Order}

Let $X$ be a real Banach space with a partial order induced by a closed cone $C$. Let $D \subseteq X$. A function $f:D \rightarrow X$ is \emph{order-preserving} if $f(x) \le f(y)$ whenever $x \le y$. The function $f$ is \emph{homogeneous} if $f(tx) = tf(x)$ for all $t > 0$ and $x \in D$ and \emph{subhomogeneous} if $f(tx) \le tf(x)$ for all $t \ge 1$ and $x \in D$. 

If $C$ has nonempty interior and $f: \inter{C} \rightarrow \inter{C}$ is order-preserving, then $f$ is nonexpansive with respect to Thompson's metric if and only if $f$ is subhomogeneous (see e.g., \cite[Lemma 2.1.7]{LemmensNussbaum} where this is proved in finite dimensions, but the same proof applies to infinite dimensional cones as well).
If $f(x) \ge \alpha x$ for some $x \in \inter{C}$, then we will say that $x$ is a \emph{sub-eigenvector} of $f$ with \emph{sub-eigenvalue} $\alpha$.  Similarly, if $f(y) \le \beta y$ for some $y \in \inter{C}$, then $y$ is a \emph{super-eigenvector} of $f$ with \emph{super-eigenvalue} $\beta$.  

The \emph{upper Collatz-Wielandt number} for any function $f:\inter{C} \rightarrow C$ is
$$r(f) = \inf_{x \in \inter{C}} M(f(x)/x)$$
and the \emph{lower Collatz-Wielandt number} for a function $f:\inter{C} \rightarrow \inter{C}$ is 
$$\lambda(f) = \sup_{x \in \inter{C}} m(f(x)/x).$$
Observe that $r(f)$ is the infimum of the super-eigenvalues of $f$, and $\lambda(f)$ is the supremum of the sub-eigenvalues of $f$. 

The main result of this section is:

\begin{theorem} \label{thm:bounded}
Let $C$ be a normal, closed cone with nonempty interior in a Banach space. Let $f: \inter{C} \rightarrow \inter{C}$ be order-preserving and subhomogeneous. If $C$ is regular or if $f$ is $\gamma$ or $\tau$-condensing, then the following are equivalent. 
\begin{enumerate}[(a)]
\item \label{item:Fix} $\Fix(f)$ is nonempty and bounded in $(\inter{C}, d_T)$.  
\item \label{item:illum} There exist $x, y \in \inter{C}$ such that $f(x) \gg x$ and $f(y) \ll y$. 
\item \label{item:lambdaR} $\lambda(f) > 1$ and $r(f)< 1$.  
\end{enumerate}
\end{theorem}

Before proving Theorem \ref{thm:bounded}, we gather a few minor lemmas.
The first lemma is known (see e.g., \cite[Theorem 19.1]{Deimling}). The proof is simple, so we include it here for convenience.

\begin{lemma} \label{lem:regular}
Let $C$ be a normal, closed cone in a Banach space $X$.  Let $D$ be a closed subset of $X$ and let $f:D \rightarrow D$ be order-preserving and continuous. Assume either that $C$ is regular or $f$ is $\gamma$ or $\tau$-condensing. If $f^k(x) \in D$ is an increasing (or decreasing) sequence that is bounded above (below) by $y \in D$, then $f^k(x)$ converges to a fixed point of $f$ in $D$.  In particular, if $[x,y] \subseteq D$ is a nonempty order interval and $f([x,y]) \subseteq [x,y]$, then $f$ has a fixed point in $[x,y]$.  
\end{lemma}

\begin{proof}
If $C$ is regular, then the sequence $f^k(x)$ converges by definition. If $f$ is either $\gamma$ or $\tau$-condensing, then $\{f^k(x) : k \in \N\}$ has compact closure by property (P1) and therefore $f^k(x)$ has a limit point in $D$. Since $f^k(x)$ is increasing (or decreasing), it converges to that limit point. In either case, the continuity of $f$ guarantees that the limit of $f^k(x)$ is a fixed point. If $[x,y] \subseteq D$ is a nonempty order-interval and $f([x,y]) \subseteq [x,y]$, it follows that $f(x) \ge x$ and therefore $f^k(x)$ is an increasing sequence bounded above by $y$.  Therefore $f^k(x)$ converges to a fixed point in $[x,y]$. 
\end{proof}


\begin{lemma} \label{lem:contraction}
Let $C$ be a normal, closed cone with nonempty interior in a Banach space. Let $f: \inter{C} \rightarrow \inter{C}$ be order-preserving and subhomogeneous. In addition, suppose either that $C$ is regular or $f$ is $\gamma$ or $\tau$-condensing.  If $u$ is a fixed point of $f$ and $\Fix(f)$ is bounded in Thompson's metric, then for any $R > r > 0$ such that $\Fix(f) \subset B_r(u)$, there is a $k \in \N$ large enough so that 
$$f^k(\cl{B_R(u)}) \subset B_r(u).$$
In particular, if $f$ has a unique fixed point $u$, then $f^k(x)$ converges to $u$ for all $x \in \inter{C}$.
\end{lemma}
\begin{proof}
Recall by \eqref{balls} that $\cl{B_R(u)} = [e^{-R} u, e^R u]$. Since $u$ is a fixed point and $f$ is subhomogeneous, it follows that $f(e^{-R}u) \ge e^{-R}u$ and $f(e^R u) \le e^R u$.  By Lemma \ref{lem:regular}, the sequences $f^j(e^{-R}u)$ and $f^j(e^R u)$ converge to fixed points of $f$.  Therefore we can choose a $k$ large enough so that both $f^k(e^{-R}u)$ and $f^k(e^R u)$ are contained in $B_r(u)$. Since $f^k$ is order-preserving, it follows that
$$f^k(\cl{B_R(u)}) = f^k([e^{-R} u, e^R u]) \subseteq [f^k(e^{-R} u), f^k(e^R u)]$$
which is contained in $B_r(u).$
\end{proof}

\begin{remark}
The observation in Lemma \ref{lem:contraction} that if $f$ has a unique fixed point $u \in \inter{C}$, then $f^k(x)$ converges to $u$ for all $x \in \inter{C}$ was made in \cite[Theorem 6.6]{Krasnoselskii64}, although with the stronger assumption that $f$ is compact when $C$ is not regular.
\end{remark}

This next lemma shows that there is an order relation between sub-eigenvectors and super-eigenvectors. 
\begin{lemma} \label{lem:subeig}
Let $C$ be a closed cone with nonempty interior in a Banach space.  Let $f: \inter{C} \rightarrow \inter{C}$ be order-preserving and subhomogeneous. If $f(x) \ge \alpha x$ and $f(y) \le \beta y$ where $\alpha > \beta$, then $x \ll y$.   
\end{lemma}
\begin{proof}
Suppose that $y - x \notin \inter{C}$.  Then there is a maximal $0 < t \le 1$ such that $y - tx \in C$.  For that $t$, 
$$\beta y \ge f(y) \ge f(tx) \ge t f(x) \ge t \alpha x.$$
Therefore $y - (\alpha/\beta) tx \in C$ which contradicts the maximality of $t$. 
\end{proof}

%

\begin{proof}[Proof of Theorem \ref{thm:bounded}]
\ref{item:Fix}$\Rightarrow$\ref{item:illum}. Choose $u \in \Fix(f)$. Choose any $R > r > 0$ with $r$ large enough so that $\Fix(f)$ is contained in the open Thompson metric ball $B_r(u)$. By Lemma \ref{lem:contraction}, there exists $k \in \N$ such that $f^k(\cl{B_R(u)}) \subset B_r(u)$.
Let 
$g = (1+\epsilon)^{-1} f \text{ and }h = f + \epsilon u$
where $\epsilon = e^{(R-r)/k} - 1$.  Both $g$ and $h$ are order-preserving, subhomogeneous functions on $\inter{C}$. If $f$ is $\gamma$-condensing, then so are both $g$ and $h$ by property (P2) of $\gamma$. Similarly, if $f$ is $\tau$-condensing, then so are $g$ and $h$ by property (P3) of $\tau$. 
A quick induction argument shows that 
$$g^k(x) \ge (1+\epsilon)^{-k} f^k(x) \ge e^{r-R} f^k(x)$$
for all $x \in \inter{C}$.  Since $h(x) = f(x) + \epsilon u = f(x) + \epsilon f(u) \le (1+\epsilon) f(x)$ for all $x \ge u$, we can use a similar induction argument to show that
$$h^k(x) \le (1+\epsilon)^k f^k(x) \le e^{R-r} f^k(x)$$
for all $x \ge u$. In particular, these inequalities imply that 
$$g^k(x) \ge e^{r-R} f^k(e^{-R} u) \ge e^{-R} u$$
when $x \ge e^{-R}u$, and 
$$h^k(x) \le e^{R-r} f^k(e^{R} u) \le e^R u$$
when $u \le x \le e^{R} u$.  

By the above inequalities, $g^{jk}(u) \ge e^{-R}u$ for all $j \in \N$.  Since the sequence $g^j(u)$ is decreasing, it follows that $g^j(u) \ge e^{-Ru}$ for all $j \in \N$. Therefore $g^j(u)$ converges to a fixed point $x$ of $g$ in $[e^{-R} u, u]$ by Lemma \ref{lem:regular}.  Similarly $h^{jk}(u) \le e^R u$ for all $j \in \N$. Since $h^j(u)$ is increasing, it must be bounded above by $e^R u$, and so it converges to a fixed point $y$ of $h$ in $[u, e^R u]$.  
Then $f(x) = (1+\epsilon)x \gg x$ and $f(y) = y - \epsilon u \ll y$.

\ref{item:illum}$\Rightarrow$\ref{item:Fix}. If $f(x) \gg x$ and $f(y) \ll y$, then there exist $\alpha > 1$ and $\beta < 1$ such that $f(x) \ge \alpha x$ and $f(y) \le \beta y$. So $y \gg x$ by Lemma \ref{lem:subeig}.  Since $f([x, y]) \subset [x,y]$, Lemma \ref{lem:regular} implies that $f$ has a fixed point in $[x, y]$. In fact, all fixed points of $f$ are contained in $[x,y]$ by Lemma \ref{lem:subeig} since every fixed point is both a super and sub-eigenvector of $f$ with eigenvalue one. Therefore $\Fix(f)$ is nonempty and bounded in Thompson's metric.
 
\ref{item:illum}$\Rightarrow$\ref{item:lambdaR}.  This follows immediately from the definition of $\lambda(f)$ and $r(f)$. 

\ref{item:lambdaR}$\Rightarrow$\ref{item:illum}. For any $\epsilon > 0$, there exists $x, y \in \inter{C}$ such that 
$$f(x) \ge (\lambda(f)-\epsilon)x ~~~~ \text{ and } ~~~~ f(y) \le (r(f)+\epsilon) y.$$ 
If $\lambda(f) > 1$ and $r(f) < 1$, then we can choose $\epsilon$ small enough so that $\lambda(f)-\epsilon > 1$ and $r(f) + \epsilon < 1$. Then $f(x) \gg x$ and $f(y) \ll y$.  
\end{proof}

The next result about the spectrum of order-preserving subhomogeneous maps follows immediately by applying Theorem \ref{thm:bounded} to $\mu^{-1} f$. Note that Krasnoselskii observed that the eigenvalues of a compact, order-preserving, subhomogeneous map form a continuous interval in \cite[Section 6.2]{Krasnoselskii64}.  

\begin{corollary} \label{cor:spectrum}
Let $C$ be a normal closed cone with nonempty interior in a Banach space. Let $f: \inter{C} \rightarrow \inter{C}$ be order-preserving and subhomogeneous. In addition, assume one of the following: (i) $C$ is regular, (ii) $\lambda(f) f$ is $\gamma$-condensing, or (iii) $f$ is $\tau$-condensing.  If $r(f) < \lambda(f)$, then for every $r(f) < \mu < \lambda(f)$, the set of eigenvectors $\{x \in \inter{C} : f(x) = \mu x \}$ is nonempty and bounded in $(\inter{C},d_T)$.  
\end{corollary}

\begin{remark}
The conditions of Theorem \ref{thm:bounded} are never satisfied if the map $f$ is also homogeneous.  If a homogeneous map has a fixed point $x \in \inter{C}$, then $\Fix(f)$ cannot be bounded since it contains the ray $\{tx : t > 0\}$. Note also that $\lambda(f) \le r(f)$ whenever $f$ is homogeneous. Therefore Theorem \ref{thm:bounded} does not help determine when order-preserving homogeneous maps have a fixed point (or eigenvector) in the interior of a cone.  For more information on that topic, see \cite{LemmensNussbaum} or for some recent results for order-preserving homogeneous maps on the standard cone $\R^n_{\ge 0}$, see \cite{Lins22}.   
\end{remark}

\begin{remark}
One might wonder if it is possible to give similar necessary and sufficient conditions for $\Fix(f)$ to be nonempty and bounded without the assumption that $f$ is order-preserving. The answer is yes for the standard cone $\R^n_{\ge 0}$, however the conditions are somewhat more complicated. The entrywise logarithm function is an isometry from $(\R^n_{> 0},d_T)$ onto $\R^n$ with the supremum norm $\|x\|_\infty = \max_{1 \le i \le n} |x_i|$. Necessary and sufficient conditions for the fixed point set of a nonexpansive map on a finite dimensional normed space to be bounded and nonempty are given in \cite{LLN16}. Verifying that a $\|\cdot\|_\infty$-nonexpansive map on $\R^n$ has a nonempty bounded fixed point set requires confirming $2^n$ inequalities similar to condition \ref{item:illum} of Theorem \ref{thm:bounded}. It is an open question whether similar conditions could be given for $d_T$-nonexpansive maps on other cones. The fact that condition \ref{item:illum} of Theorem \ref{thm:bounded} only requires confirming two inequality conditions demonstrates how strong the order-preserving property is. 
\end{remark}


Theorem \ref{thm:bounded} raises the question of how to compute the upper and lower Collatz-Wielandt numbers of an order-preserving, subhomogeneous function $f$. Here we make some observations that can help.  

Let $C$ be a closed cone with nonempty interior in a Banach space $X$, and suppose that $f:\inter{C} \rightarrow \inter{C}$ is order-preserving and subhomogeneous. The \emph{recession map} of $f$ is the function defined by 
$$f_\infty(x) = \lim_{t \rightarrow \infty} t^{-1} f(tx)$$
for all values $x \in \inter{C}$ where the limit exists.  
The use of recession maps to study the eigenvalues of order-preserving subhomogeneous maps can be traced back at least as far as \cite[Theorem 6.11]{Krasnoselskii64}, although there the recession maps are required to be linear. See \cite{CaLiSt19} for a more recent example where the recession map is used to give conditions for the existence of fixed points of order-preserving subhomogeneous maps on $\R^n_{>0}$.

\begin{proposition} \label{prop:recession}
Let $C$ be a closed, regular cone with nonempty interior in a Banach space. Let $f: \inter{C} \rightarrow \inter{C}$ be order-preserving and subhomogeneous. Then the recession map $f_\infty$ exists for every $x \in \inter{C}$, and $f_\infty:\inter{C} \rightarrow C$ is order-preserving and homogeneous.  Furthermore $r(f) = r(f_\infty)$.  
\end{proposition}
\begin{proof}
Since $f$ is subhomogeneous, $t^{-1}f(tx) \le s^{-1}f(sx)$ whenever $0 < s < t$.  Since $C$ is regular, it follows that $f_{\infty}(x) = \lim_{t \rightarrow \infty} t^{-1} f(tx)$ exists for every $x \in \inter{C}$. For any $\lambda > 0$, can use the substitution $s = t \lambda$ to see that
$$f_\infty(\lambda x) = \lim_{t \rightarrow \infty} \tfrac{1}{t} f(t\lambda x) = \lim_{s \rightarrow \infty} \tfrac{\lambda}{s} f(sx) = \lambda f_\infty(x)$$
so $f_\infty$ is homogeneous. If $x \le y$ in $\inter{C}$, then $t^{-1}f(tx) \le t^{-1}f(ty)$ for all $t > 0$, so we have $f_\infty(x) \le f_\infty(y)$.  

Since $f_\infty(x) \le f(x)$ for all $x \in \inter{C}$, it follows from the definition that $r(f_\infty) \le r(f)$.  Choose a sequence $x_n \in \inter{C}$ such that $f_\infty(x_n) \le \beta_n x_n$ where the each $\beta_n > 0$ and the sequence $\beta_n$ converges to $r(f_\infty)$.  For each $x_n$, there exists $t_n \ge 1$ such that $t_n^{-1} f(t_n x_n) \le (\beta_n + \tfrac{1}{n}) x_n$.  Then since the sequence $\beta_n + \tfrac{1}{n}$ converges to $r(f_\infty)$ it follows that $r(f) = r(f_\infty)$. 
\end{proof}

If $C$ is a normal closed cone with nonempty interior and $f:\inter{C} \rightarrow \inter{C}$ is order-preserving and homogeneous, then the upper Collatz-Wielandt number $r(f)$ is the same as the \emph{partial cone spectral radius} 
$$r_{\inter{C}}(f) = \lim_{k \rightarrow \infty} \|f^k(x)\|^{1/k},$$
where $x \in \inter{C}$ is arbitrary and the choice of $x$ does not affect the value of the limit \cite[Lemma 4.2 and Theorem 4.6]{LLNW18}. Unfortunately, a recession map $f_\infty$ may not send $\inter{C}$ into itself (see \eqref{dagger1} for an example). If $f: \inter{C} \rightarrow C$ is order-preserving and homogeneous, but $f(\inter{C})$ is not contained in $\inter{C}$, then it is still possible to use an iterative method to calculate the upper Collatz-Wielandt number.  The key idea is to combine the iterative formula above with a small perturbation of $f$.  

\begin{proposition} \label{prop:iteration}
Let $C$ be a closed normal cone with nonempty interior in a Banach space. Fix $u \in \inter{C}$ and let $\|\cdot\|_u$ denote the corresponding order-unit norm defined by \eqref{order-unit}.  If $f: \inter{C} \rightarrow C$ is order-preserving and homogeneous, then  
\begin{align*}
r(f) &= \lim_{k \rightarrow \infty} \|(f+ \id)^k(u)\|_u^{1/k} - 1 \\
&= \inf_{k > 0} \|(f+\id)^k(u)\|_u^{1/k} - 1\\  
&= \lim_{k \rightarrow \infty} \|(f+\id)^k(u)\|^{1/k} - 1.
\end{align*}
In particular, $r(f) < 1$ if and only if there is a $k \in \N$ such that $(f+\id)^k(u) \ll 2^k u$.
\end{proposition}
\begin{proof}
It is clear from the definition of the upper Collatz-Wielandt number that $r(f) + 1 = r(f+\id).$
Since the perturbed map $f+\id$ is order-preserving, homogeneous, and maps $\inter{C}$ into $\inter{C}$, \cite[Lemma 4.2 and Theorem 4.6]{LLNW18} implies that
$$r(f) + 1 = \lim_{k \rightarrow \infty} \|(f + \id)^k(x)\|^{1/k}.$$
for any $x \in \inter{C}$.  Furthermore, in the proof of \cite[Theorem 4.6]{LLNW18} it is shown that
$$\lim_{k \rightarrow \infty} \|(f+ \id)^k(u)\|_u^{1/k} = \inf_{k > 0} \|(f+\id)^k(u)\|_u^{1/k} = \lim_{k \rightarrow \infty} \|(f+\id)^k(u)\|^{1/k}.$$
Therefore $r(f) < 1$ if and only if $\inf_{k > 0} \|(f+\id)^k(u)\|_u^{1/k} < 2$.  This happens if and only if there is a $k > 0$ such that $\|(f+\id)^k(u)\|_u < 2^k$ which is equivalent to $(f+\id)^k(u) \ll 2^k u$. 
\end{proof}

We now consider how to compute the lower Collatz-Wielandt number. Some cones have an order-reversing, bijective $d_T$-isometry $L: \inter{C} \rightarrow \inter{C}$. By \emph{order-reversing}, we mean that $L(x) \le L(y)$ for all $x \ge y \gg 0$. The inverse function on a symmetric cone is an order-reversing $d_T$-isometry, as is the operator $f(x) \mapsto 1/f(x)$ for the cone of positive functions in many function spaces. For finite dimensional cones, Walsh \cite{Walsh18} proved that an order-reversing, bijective $d_T$-isometry exists if and only if $C$ is a symmetric cone.  When there is an order-reversing isometry, we can make the following simple observation which is an immediate corollary of Proposition \ref{prop:recession}.

\begin{lemma} \label{lem:recession2}
Let $C$ be a regular, closed cone with nonempty interior in a Banach space and let $f: \inter{C} \rightarrow \inter{C}$ be order-preserving and subhomogeneous. If there is an order-reversing, bijective $d_T$-isometry $L:\inter{C} \rightarrow \inter{C}$, then 
$$\lambda(f) = r((L f L)_\infty)^{-1}$$ 
where $(LfL)_\infty$ is the recession map of the composition $L f L$.  
\end{lemma}

\section{Analytic maps and uniqueness of fixed points} \label{sec:Analytic}

Let $X, Y$ be real Banach spaces and let $U$ be an open subset of $X$.  A function $f:U \rightarrow Y$ is \emph{real analytic} if for every $x \in U$, there is an $r > 0$ and continuous symmetric $n$-linear forms $A_n: X^n \rightarrow Y$ such that $\sum_{n = 1}^\infty \|A_n\| r^n < \infty$ and 
$$f(x + h) = f(x) + \sum_{n = 1}^\infty A_n(h^n)$$ 
for all $h$ in a neighborhood of $0$ in $X$.    

A \emph{real analytic variety} is a set of common zeros of a finite collection of real analytic functions on an open domain $U \subseteq X$.  The following proposition is based on a theorem of Sullivan about the structure of real analytic varieties.  

\begin{proposition} \label{prop:analytic}
Let $V$ be a nonempty real analytic variety defined in an open subset of a finite dimensional real normed space.  If $V$ is compact and contractible, then $V$ consists of a single point. 
\end{proposition}

\begin{proof}
Let $k = \dim V$ and suppose by way of contradiction that $k > 0$. A theorem of Sullivan \cite[Corollary 2]{Sullivan71} asserts that V is locally homeomorphic to the topological cone over a polyhedron with even Euler characteristic. As an immediate consequence, the sum of all $k$-simplices in a triangularization of $V$ is a mod 2 cycle \cite[see comments after Corollary 2]{Sullivan71}. Since there are no simplices of dimension greater than $k$ in a triangularization of $V$, the sum of all $k$-simplices is not a boundary.  Therefore the homology group $H_k(V, \mathbb{Z}_2) \ne \varnothing$.  This means that $V$ cannot be contractible, a contradiction.
\end{proof}

\begin{theorem} \label{thm:analytic1}
Let $C$ be a closed cone with nonempty interior in a finite dimensional real normed space. Let $f: \inter{C} \rightarrow \inter{C}$ be real analytic and $d_T$-nonexpansive. Then $f$ has a unique fixed point in $\inter{C}$ if and only if $\Fix(f)$ is a nonempty and bounded subset of $(\inter{C},d_T)$. 
\end{theorem}

\begin{proof}
We only need to prove that $\Fix(f)$ nonempty and bounded implies that $f$ has a unique fixed point since the converse is obvious. Observe that $\Fix(f)$ is the zero set of the real analytic function $f - \id$, so it is a real analytic variety.  Since $X$ is finite dimensional and we are assuming that $\Fix(f)$ is bounded in $(\inter{C}, d_T)$, it follows that $\Fix(f)$ is compact.  We know by Corollary \ref{cor:contractible} that $\Fix(f)$ is contractible.  Therefore Proposition \ref{prop:analytic} implies that $\Fix(f)$ consists of a single point.   
\end{proof}

We can say more when $f$ is also order-preserving. The proof of the next theorem uses a different technique, which allows us to remove the assumption that the cone is finite dimensional. 

\begin{theorem} \label{thm:analytic2}
Let $C$ be a closed, normal cone with nonempty interior in a Banach space. Let $f: \inter{C} \rightarrow \inter{C}$ be real analytic, order-preserving, and subhomogeneous. Suppose that $C$ is regular or $f$ is $\gamma$ or $\tau$-condensing. Then the following are equivalent.
\begin{enumerate}[(a)]
\item \label{item:Fix2} $\Fix(f)$ is nonempty and bounded in $(\inter{C},d_T)$. 
\item \label{item:illum2} There exist $x, y \in \inter{C}$ such that $f(x) \gg x$ and $f(y) \ll y$. 
\item \label{item:lambdaR2} $\lambda(f) > 1$ and $r(f)< 1$.  
\item \label{item:unique} $f$ has a unique fixed point in $\inter{C}$.  
\item \label{item:attractor} There is a $u \in \inter{C}$ such that $\lim_{k \rightarrow \infty} f^k(x) = u$ for all $x \in \inter{C}$.
\end{enumerate}
\end{theorem}

\begin{proof}
Theorem \ref{thm:bounded} proved the equivalence of \ref{item:Fix2}, \ref{item:illum2}, and \ref{item:lambdaR2}. It is obvious that \ref{item:unique} implies \ref{item:Fix2}. Here we will prove that \ref{item:Fix2} implies \ref{item:unique} and that \ref{item:unique} and \ref{item:attractor} are equivalent.  

\ref{item:Fix2}$\Rightarrow$\ref{item:unique}. 
Choose $u \in \Fix(f)$, and choose $R > r>0$ with $r$ large enough so that $\Fix(f) \subset B_r(u)$.  By Lemma \ref{lem:contraction} there is a $k \in \N$ such that $f^k(\cl{B_R(u)}) \subset B_r(u)$. In particular, by \eqref{balls}, we have $f^k(e^R u) \ll e^r u$ and $f^k(e^{-R} u) \gg e^{-r} u$. 

Since the composition of two real analytic functions is real analytic \cite{Whittlesey65}, it follows that $f^k$ is real analytic. Now, choose $\phi \in C^* \bs \{0\}$, and consider the real-valued function $g(t) = t^{-1} \phi(f^k(tu))$ which is defined on the real interval $(0,\infty)$. Suppose $0 < s < t$. 
By \eqref{functionals}, 
$$ \log \frac{\phi(f^k(tu))}{\phi(f^k(su))} \le d_T(f^k(tu), f^k(su)) \le d_T(tu,su) = \log \left(\frac{t}{s}\right).$$
It follows that $g(t) \le g(s)$, so $g$ is monotone decreasing. If $g(t) = g(1) = \phi(u)$ for any $t > 0$ other than 1, then $g$ will be constant on the interval between $t$ and $1$. Since $g$ is real analytic, that would imply that $g$ is constant on all of $(0, \infty)$. That cannot be the case, however, since $f^k(e^R u) \ll e^r u$, which implies that 
$$g(e^R) \le e^{-R} \phi(e^r u) = e^{r-R} \phi(u) < \phi(u) = g(1).$$  
From this we conclude that $g$ is strictly decreasing and therefore $\phi(f^k(tu)) < t\phi(u)$ for all $t > 1$ and $\phi(f^k(tu)) > t\phi(u)$ for all $t < 1$.  These inequalities are true for all $\phi \in C^* \bs \{0\}$.  

Note that $x \in \inter{C}$ if and only if $\phi(x) > 0$ for all $\phi \in C^* \bs \{0\}$ \cite[Proposition 19.3(b)]{Deimling}. Since $\phi(tu - f^k(tu)) > 0$ for all $\phi \in C^* \bs \{0\}$ and $t > 1$, we have $f^k(tu) \ll tu$ when $t > 1$.  Likewise, $f^k(tu) \gg tu$ when $0 < t < 1$. So for every $t > 1$, $f^k$ maps the closed Thompson metric ball $\cl{B_{\log t}(u)} = [t^{-1} u, t u]$ into its interior. This implies that $d_T(f^k(x),u) < d_T(x,u)$ for all $x \in \inter{C}$. From this, we conclude that $u$ is the only fixed point of $f^k$, and so it must also be the only fixed point of $f$ as well.  

\ref{item:unique}$\Leftrightarrow$\ref{item:attractor}. If $f$ has a unique fixed point $u \in \inter{C}$, then Lemma \ref{lem:contraction} implies that $\lim_{k \rightarrow \infty} f^k(x) = u$ for all $x \in \inter{C}$.  Conversely, if $f^k(x)$ converges to $u \in \inter{C}$ for all $x \in \inter{C}$, then $u$ is a fixed point of $f$ since $f$ is continuous and $u$ is unique since the iterates of all other $x \in \inter{C}$ converge to $u$.
\end{proof}

\begin{remark}
It is known that if $X$ is a real Banach space with the fixed point property and $f: X \rightarrow X$ is real analytic and nonexpansive, then $f$ has a unique fixed point if and only if $\Fix(f)$ is bounded and nonempty \cite[Theorem 5.3]{Lins22}.
\end{remark}


\section{Application to nonlinear matrix equations} \label{sec:matrix}

Let $\mathcal{H} \subset \C^{n \times n}$ denote the set of $n$-by-$n$ Hermitian matrices with complex entries. Let $\mathcal{P}$ denote the cone of positive semidefinite matrices in $\mathcal{H}$. Let $I$ denote the $n$-by-$n$ identity matrix and let $\|\cdot\|$ denote the spectral norm on $\mathcal{H}$. Note that $\inter{\mathcal{P}}$ is the set of all positive definite matrices.  Let $L(X) = X^{-1}$ for any $X \in \inter{\mathcal{P}}$. The following nonlinear function $f: \inter{\mathcal{P}} \rightarrow \mathcal{P}$ was studied in \cite[Section 8]{LaLi12}:
$$f(X) = A + N^* (B + X^{-1})^{-1} N$$
where $A, B \in \mathcal{P}$ and $N \in \C^{n \times n}$. Any fixed point of $f$ is a solution to the discrete algebraic Riccati equation $X = A + N^* (B + X^{-1})^{-1} N$.
Assume that $A + N^*N \in \inter{\mathcal{P}}$ so that $f(\inter{\mathcal{P}}) \subseteq \inter{\mathcal{P}}$. Unlike \cite{LaLi12}, we do not assume that $A$ and $B$ are positive definite, only positive semidefinite which is enough to guarantee that $f$ is order-preserving and subhomogeneous on $\inter{\mathcal{P}}$.  As a consequence of \cite[Theorem 8.1]{LaLi12}, $f$ is a strict $d_T$-contraction on $\inter{\mathcal{P}}$ when $A, B$ are positive definite.  With our weaker assumptions, $f$ may not be a strict $d_T$-contraction. 
In this section, we will demonstrate how to calculate the recession maps $f_\infty$ and $(LfL)_\infty$. Using the recession maps, we will give simple sufficient conditions for $f$ to have a unique positive definite fixed point, despite the fact that $f$ may not be a strict contraction.  

First, we note the following.  

\begin{lemma} \label{lem:schur}
Let $A \in \mathcal{P}$ and let $B(t) \in \mathcal{P}$ for all $t > 0$ and suppose that $\lim_{t \rightarrow \infty} B(t) = B(\infty) \in \mathcal{P}$ exists and $A + B(t) \in \inter{\mathcal{P}}$ for all $t \in (0,\infty]$.  Then 
$$\lim_{t \rightarrow \infty} (tA + B(t))^{-1} = (Q_A B(\infty) Q_A)^\dagger$$
where $Q_A \in \mathcal{H}$ is the orthogonal projection onto the nullspace of $A$ and $M^\dagger$ denotes the Moore-Penrose pseudoinverse of a matrix $M$. 
\end{lemma}

\begin{proof}
Recall the Schur complement formula for the inverse of a partitioned matrix \cite[Section 0.7.3]{HornJohnson}:
\begin{equation} \label{matrix-inverse}
\begin{bmatrix} M_{11} & M_{12} \\ M_{21} & M_{22} \end{bmatrix}^{-1} = \begin{bmatrix} S_1^{-1} & -M_{11}^{-1} M_{12} S_2^{-1} \\ -S_2^{-1} M_{21} M_{11}^{-1} & S_2^{-1} \end{bmatrix}
\end{equation}
where $S_1 = M_{11}-M_{12} M_{22}^{-1} M_{21}$ and $S_2 = M_{22} - M_{21} M_{11}^{-1} M_{12}$. 
Note that all of the inverses in this formula are well-defined when the partitioned matrix is positive definite \cite[Theorem 7.7.6]{HornJohnson}.

We can choose a basis so that $M(t) = tA+B(t)$ can be expressed as a partitioned matrix:
$$M(t) = \begin{bmatrix} 
t A_{11} + B(t)_{11} & B(t)_{12} \\
B(t)_{21} & B(t)_{22}
\end{bmatrix}$$
where 
$$A = \begin{bmatrix}
A_{11} & 0 \\ 
0 & 0 \end{bmatrix}$$
has $A_{11}$ positive definite. Since $M(t)$ is positive definite for all $t > 0$, we can apply \eqref{matrix-inverse} to $M(t)$.  Then by inspection, the limit of $M(t)^{-1}$ as $t \rightarrow \infty$ is 
$$\begin{bmatrix} 
0 & 0 \\ 
0 & B(\infty)_{22}^{-1}
\end{bmatrix} = (Q_A B(\infty) Q_A)^\dagger. $$
\end{proof}

Now, we use Lemma \ref{lem:schur} to compute $f_\infty$ and $(LfL)_\infty$. For any $X \in \inter{\mathcal{P}}$, we have  
\begin{align} \label{dagger1}
f_\infty(X) &= \lim_{t \rightarrow \infty} t^{-1} (A + N^*(B+(tX)^{-1})^{-1}N) \notag \\ 
&= \lim_{t \rightarrow \infty} N^*(tB+X^{-1})^{-1}N \\ 
&= N^*(Q_B X^{-1} Q_B)^\dagger N & (\text{by Lemma }\ref{lem:schur}) \notag
\end{align}
and
\begin{align} \label{dagger2}
(LfL)_\infty(X) &= \lim_{t \rightarrow \infty} t^{-1} (A + N^*(B+tX)^{-1}N)^{-1} \notag \\
&= \lim_{t \rightarrow \infty} (tA + N^*(t^{-1}B+X)^{-1}N)^{-1} \\
&= (Q_A N^*X^{-1} N Q_A)^\dagger & (\text{by Lemma }\ref{lem:schur}) \notag 
\end{align}
where $Q_A$ and $Q_B$ are the orthogonal projections onto the nullspaces of $A$ and $B$, respectively.


\begin{proposition}
Let $f(X) = A + N^*(B + X^{-1})^{-1}N$ where $A, B \in \mathcal{P}$ and $N \in \C^{n \times n}$ satisfies $A+ N^*N \in \inter{\mathcal{P}}$. Let $Q_A$ and $Q_B$ denote the orthogonal projections onto the nullspaces of $A$ and $B$ respectively, and let $f_\infty$ and $(LfL)_\infty$ be given by equations \eqref{dagger1} and \eqref{dagger2}. Then $f$ has a (necessarily unique) globally attracting fixed point in $\inter{\mathcal{P}}$ if and only if there are constants $k, \ell > 0$ such that 
$$\|(f_\infty + \id)^k(I)\| < 2^k  \text{ and }\|((LfL)_\infty+\id)^\ell(I)\| < 2^\ell.$$ 
In particular, $f$ has a globally attracting fixed point if 
$$\|N^* Q_B N \| < 1 \text{ and } \| (Q_A N^*N Q_A)^\dagger \| < 1.$$
\end{proposition}

\begin{proof}
Observe that $f$ is real analytic because the inverse operator $L$ is real analytic in a neighborhood of any $X \in \inter{\mathcal{P}}$. Since $A, B \in \mathcal{P}$ and $A+N^*N \in \inter{\mathcal{P}}$, it follows that $f$ is an order-preserving, subhomogeneous map sending $\inter{\mathcal{P}}$ into itself. When $X \in \mathcal{P}$, note that $X \ll I$ if and only if $\|X \| < 1$. In particular, $\|(f_\infty + \id)^k(I)\| \le 2^k$ if and only if $(f_\infty + \id)^k(I) < 2^k I$. By Propositions \ref{prop:recession} and \ref{prop:iteration}, there is a $k > 0$ for which these equations hold if and only if $r(f) < 1$. Similarly, $\|((LfL)_\infty + \id)^\ell(I)\| \le 2^\ell$ for some $\ell > 0$ if and only if $r((LfL)_\infty) < 1$. This is equivalent to $\lambda(f) > 1$ by Lemma \ref{lem:recession2}.  Therefore Theorem \ref{thm:analytic2} says that $f$ has a globally attracting fixed point in $\inter{\mathcal{P}}$ if and only if
$$\|(f_\infty + \id)^k(I)\| < 2^k  \text{ and }\|((LfL)_\infty+\id)^\ell(I)\| < 2^\ell$$ 
for some $k, \ell \in \N$.  This condition is satisfied with $k = \ell = 1$ if 
$$\|N^* Q_B N \| < 1 \text{ and } \| (Q_A N^*N Q_A)^\dagger \| < 1.$$
\end{proof}


\bibliography{DW3}
\bibliographystyle{plain}

\end{document}